\documentclass[runningheads]{llncs}
\usepackage{amsmath}
\usepackage{amssymb}
\usepackage[applemac]{inputenc}
\usepackage[colorlinks]{hyperref}
\usepackage{url}
\usepackage{graphicx,epstopdf}
\usepackage{subcaption}
 \usepackage{epstopdf}
\usepackage[applemac]{inputenc}
\usepackage[colorlinks]{hyperref}
\usepackage{url}
\usepackage{subcaption}
\usepackage{hyperref}
\usepackage[T1]{fontenc}
\usepackage{graphicx}
\usepackage[all]{xy}
\newcommand\de{\delta}
\newcommand\ep{\varepsilon}

\newcommand\io{\iota}
\newcommand\ka{\kappa}

\newcommand\ta{\tau}

\newcommand\Si{\Sigma}

\newcommand\on{\operatorname}

\newcommand\Diff{\on{Diff}}
\renewcommand\o{\circ}
\newcommand\x{\times}

\newcommand\F{\mathcal F}
\renewcommand\S{\mathcal S}

\newcommand\C{\mathcal C}

\renewcommand\P{\mathcal P}
\newcommand\G{\mathcal G}
\newcommand\E{\mathcal E}

\begin{document}
\title{Shape spaces of nonlinear flags \thanks{Supported by FWF grant I 5015-N, Institut CNRS Pauli, University of Timi\c{s}oara and University of Lille}}

\author{Ioana~Ciuclea\inst{1}\orcidID{0000-0002-8467-2706} \and
Alice~Barbora~Tumpach\inst{2, 3}\orcidID{0000-0002-7771-6758}\and
Cornelia~Vizman\inst{1}\orcidID{0000-0002-5551-3984}}
\authorrunning{Ciuclea, Tumpach and Vizman}
\institute{ 
West University of Timi\c{s}oara,  bld. Vasile P\^{a}rvan no. 4, Timi\c{s}oara, 300223 Timi\c{s}~county, Romania
\email{ioana.ciuclea@gmail.com, cornelia.vizman@e-uvt.ro }
\and
Institut CNRS Pauli, Oskar-Morgenstern-Platz 1, 1090 Vienna, Austria \and
University of Lille, Cit\'e scientifique, 59650 Villeneuve d'Ascq, France
\email{alice-barbora.tumpach@univ-lille.fr}\\
\url{http://math.univ-lille1.fr/~tumpach/Site/home.html}}

\maketitle              % typeset the header of the contribution
\begin{abstract}

The shape space considered in this article consists of surfaces embedded in $\mathbb{R}^3$, that are decorated with curves. It is a special case of the Fr\'echet manifolds of nonlinear flags, i.e. nested submanifolds of a fixed type. 
%In our case the type is the equator embedded in the sphere.
The gauge invariant elastic metric on the shape space of surfaces involves the mean curvature and the normal deformation, i.e. the sum and the difference of the principal curvatures $\kappa_1,\kappa_2$. The proposed gauge invariant elastic metrics on the space of surfaces decorated with curves involve, in addition, the geodesic and normal curvatures  $\kappa_g,\kappa_n$ of the curve on the surface, as well as the geodesic torsion $\tau_g$.

 More precisely, we show that, with the help of the Euclidean metric, the tangent space at $(C,\Sigma)$ can be identified with  $C^\infty(C)\times C^\infty(\Sigma)$ 
 and the gauge invariant elastic metrics
 %reduce to the following
 form a 6-parameter family:
 \begin{align*}
 %\begin{array}{lll}
 \G_{(C,\Sigma)}(h_1,h_2)&=
 a_1\int_C(h_1\kappa_g+{h_2}|_C\kappa_n)^2d\ell
 &+ 
a_2\int_{\Sigma}(h_2)^2(\kappa_1-\kappa_2)^2dA\\
& +b_1\int_C(D_sh_1-{h_2}|_C\tau_g)^2d\ell
 &+
 b_2\int_{\Sigma} (h_2)^2(\kappa_1+\kappa_2)^2dA\\
&+c_1\int_C(D_s({h_2}|_C)+h_1\tau_g)^2d\ell
&+
c_2\int_{\Sigma} |\nabla h_2|^2 dA,
%\end{array}
 \end{align*}
where $h_1\in C^\infty(C),h_2\in C^\infty(\Sigma)$.

\end{abstract}
\section{Introduction}

In this paper we use the elastic metrics on parameterized curves (\cite{Mio,Srivastava}) and parameterized surfaces (\cite{JermynECCV2012}) in order to endow the shape space of  surfaces decorated with curves  with a family of Riemannian metrics. 
This shape space of decorated surfaces is an example of Fr\'echet manifold of nonlinear flags,  studied in \cite{VizHal}. 
These consist of nested submanifolds of a fixed type $S_1\xrightarrow{\iota_1}S_2\xrightarrow{\iota_2}\cdots\to S_r$, which in our case is the equator embedded in the sphere $\mathbb{S}^1\xrightarrow{\iota}\mathbb{S}^2$.

We emphasize that we do not use the quotient elastic metrics on curves and surfaces, but rather their gauge invariant relatives (see \cite{TumPres,TumPres2}). Indeed, following \cite{Tum2,Tum1} for surfaces in $\mathbb{R}^3$, and \cite{Drira} for curves in $\mathbb{R}^3$, we construct degenerate metrics on parameterized curves and surfaces by first projecting an arbitrary variation of a given curve or surface onto the space of vector fields perpendicular to the curve or surface (for the Euclidean product of $\mathbb{R}^3$) and then applying the elastic metric on this component. By construction, vector fields tangent to the curve or surface will have vanishing norms, leading to a degenerate metric on pre-shape space. However, since the degeneracy is exactly along the tangent space to the orbit of the reparameterization group, these degenerate metrics define Riemannian (i.e. non-degenerate) metrics on shapes spaces of curves and surfaces. The following advantages of this procedure can be mentioned:
\begin{itemize}
\item[$\bullet$] there is no need to compute a complicated horizontal space in order to define a Riemannian metric on shape space
\item[$\bullet$] the length of paths in shape space equals the length of any of their lifts for the corresponding degenerate metric, a property called \textit{gauge invariance} in \cite{Tum2,Tum1}.
\item[$\bullet$] the resulting Riemannian metric on shape spaces can be easily expressed in terms of geometric invariants of curves and surfaces, leading to expressions that are completely independant of parameterizations.
\end{itemize}

In this paper, we use the gauge invariant (degenerate) metrics on parameterized curves and surfaces obtained from the elastic metrics via the procedure described above in order to define Riemannian metrics on shape spaces of curves and surfaces. Then we embed the shape space of nonlinear flags consisting of surfaces decorated with curves into the Cartesian product of the shape space of curves in $\mathbb{R}^3$ with the shape space of surfaces in $\mathbb{R}^3$. The Riemannian metric obtained on the shape space of nonlinear flags can be made explicit thanks to a precise description of its tangent space (Theorem~\ref{identification}) and thanks to the geometric expressions of the metrics used on curves and surfaces, leading to a $6$-parameter family of natural Riemannian metrics (Theorem~\ref{theorem_expression_metric}).  

Note that although the space of parameterized flags is a principal fiber bundle over the shape space of flags as it is the case for parameterized curves or surfaces, in this case there is no natural complement to the tangent space to the fibers.  
Another way to express this difference is that the principal bundle of parameterized curves or surfaces in $\mathbb{R}^3$ has a canonical connection, which is not the case for the space of nonlinear flags.
%Let us also mention that the construction of the $6$-parameter family of Riemannian metric obtained in this paper can be extended without substantial changes to the case of surfaces of any genus decorated with curves and embedded into an arbitrary $3$-dimensional Riemannain manifold (see \cite{VizHal}). 

\section{Manifolds of decorated surfaces as shape spaces}
%{Mathematical Setup}

\subsection{Notations}

We will consider the \textit{shape space} of nonlinear flags consisting of pairs $(C, \Sigma)$  such that $C$ is a curve on the surface $\Sigma$ embedded in $\mathbb{R}^3$. We will restrict our attention to surfaces of genus $0$, and simple curves (the complement to the curve in the surface has only two connected components), but our construction can be extended without substantial changes to surfaces of genus $g$ and to a finite number of curves. 
We denote the space of all such nonlinear flags by
$$
\mathcal{F} := \operatorname{Flag}_{\mathbb{S}^1\stackrel{\iota}{\hookrightarrow} \mathbb{S}^2}(\mathbb{R}^3).
$$
The general setting for the Fr\'echet manifolds on nonlinear flags of a given type $S_1\xrightarrow{\iota_1}S_2\xrightarrow{\iota_2}\cdots\to S_r$
can be found in \cite{VizHal}.

Examples of elements $(C, \Sigma)\in \mathcal{F}$ are given in Fig.~1.
\begin{figure*}[ht!]
\includegraphics[width = 0.4 \linewidth]{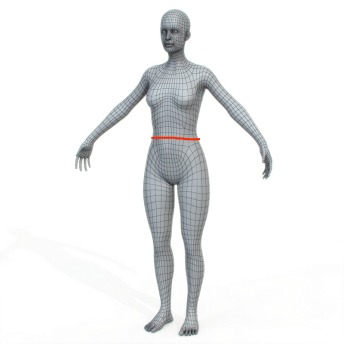}
\includegraphics[width = 0.2 \linewidth]{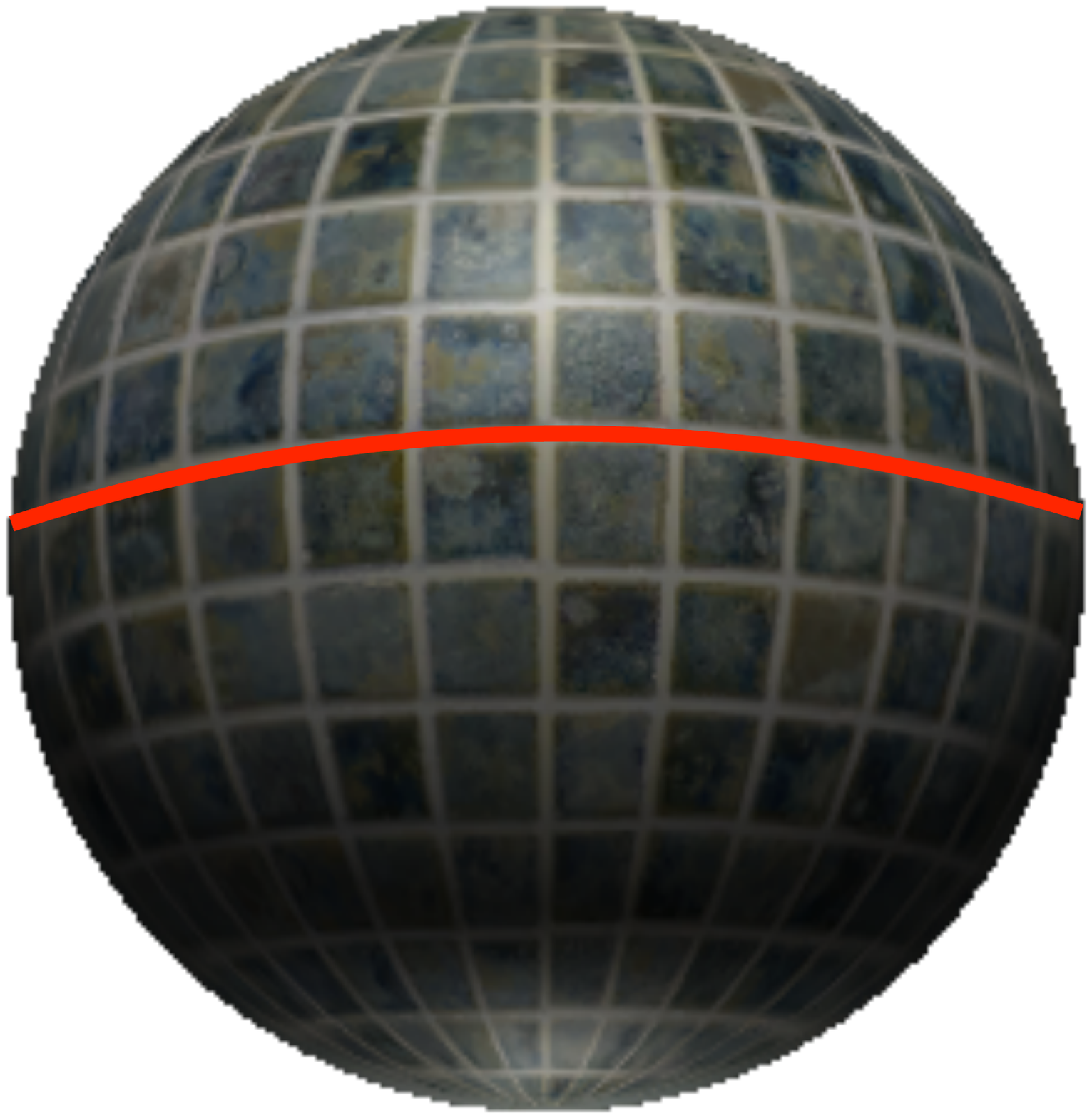}
\includegraphics[width = 0.3 \linewidth]{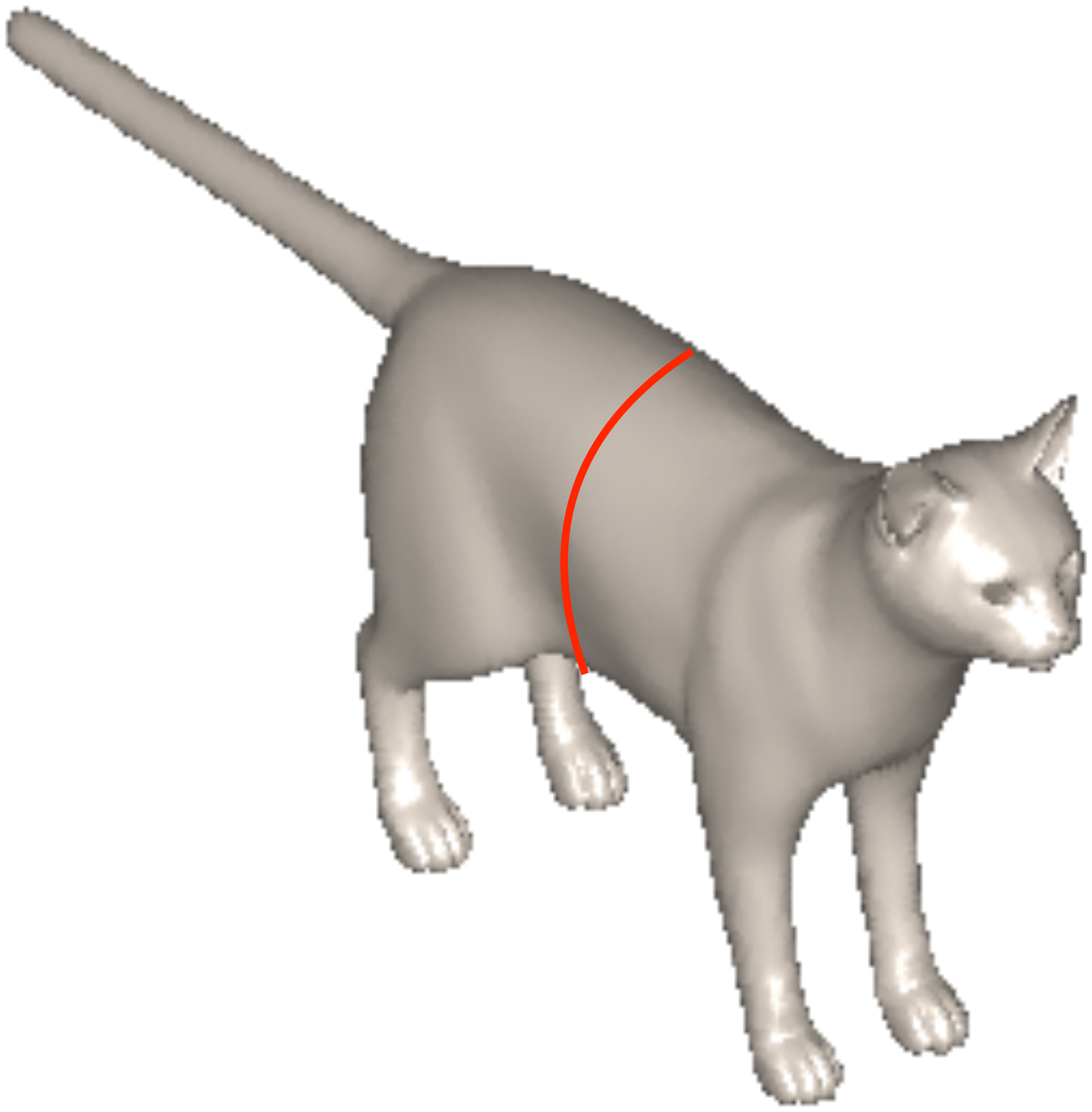}
\caption{Examples of elements in the shape space of nonlinear flags.}
\end{figure*}\label{Examples}

We will be interested in deforming flags. To this aim, we will consider a reference flag to be the pair $(\mathbb{S}^1, \mathbb{S}^2)$ consisting in the unit sphere $\mathbb{S}^2$ and the unit circle $\mathbb{S}^1$ embedded as the equator $\iota~:\mathbb{S}^1\hookrightarrow \mathbb{S}^2$.
We will represent a general flag $(C, \Sigma)$ using an embedding $F~:\mathbb{S}^2\rightarrow\mathbb{R}^3$ such that 
the image of the restriction  $F\circ \iota$ of $F$ to the equator is $C$. 
The pair $(F\circ\iota, F)$ is also called a \textit{parameterization} of the flag $(C, \Sigma)$. The space of parametrized flags is 
$$\mathcal{P} := 
\{F~:\mathbb{S}^2\rightarrow \mathbb{R}^3, F \textrm{~is~an~embedding}\}.$$
It is called the \textit{pre-shape space} of flags since objects with same shape but different
 parameterizations correspond to different points in $\mathcal{P}$. 
The set $\mathcal{P}$ is  a manifold, as an open subset of the linear space $\mathcal{C}^{\infty}(\mathbb{S}^2, \mathbb{R}^3)$ of smooth functions from 
$\mathbb{S}^2$ to $\mathbb{R}^3$. The tangent space to $\mathcal{P}$ at $F$, denoted by $T_F\mathcal{P}$, is therefore just $\mathcal{C}^{\infty}(\mathbb{S}^2, \mathbb{R}^3)$.

There is a natural projection $\pi$ from the space of parameterized flags $\mathcal{P}$ onto the space of  flags $\mathcal{F}$ given by 
\begin{equation}\label{pi}
\pi(F) = ((F\circ \iota)(\mathbb{S}^1), F(\mathbb{S}^2)).
\end{equation}
%We will denote by $\pi_1(F)$ the curve $C := (F\circ \iota)(\mathbb{S}^1)$ and by $\pi_2(F)$ the surface $\Sigma:= F(\mathbb{S}^2)$.

%%%

\subsection{Shape space of nonlinear flags as quotient space}

Since we are only interested in unparameterized nonlinear flags, 
we would like to identify pairs of parameterized curves and surfaces that can be related through reparameterization. 
The reparametrization group $G$ is the group of diffeomorphisms $\gamma$ of $\mathbb{S}^2$ which restrict to a diffeomorphism of the equator $\iota~: \mathbb{S}^1\hookrightarrow \mathbb{S}^2$~:
$$
\begin{array}{ll}
G 
 &= \{\gamma\in\operatorname{Diff}(\mathbb{S}^2)~:~\gamma\circ\iota = \iota\circ\bar{\gamma} \textrm{ for some } \bar{\gamma}\in \operatorname{Diff}(\mathbb{S}^1)\}.
\end{array}
$$ 
The group $G$ is an infinite-dimensional Fr\'echet Lie group whose Lie algebra is the space of vector fields on $\mathbb{S}^2$ whose restriction to the equator is tangent to the equator.
The right action of $G$ on $\mathcal{P}$ is given by $F\cdot \gamma := F\circ \gamma$.
It's a principal action for the principal $G$-bundle $\pi:\P\to\F$.

The elements in $\mathcal{P}$ obtained by following a fixed 
parameterized flag $F\in\mathcal{P}$ when 
acted on by all elements of $G$ is called the $G$-\textit{orbit} of $F$ or the \textit{equivalence class} of $F$ under the action of $G$,
and will be denoted by $[F]$. 
The orbit of $F\in \mathcal{P}$ is characterized by the surface $ \Sigma:= F(\mathbb{S}^2) $ and the curve $C:= (F\circ \iota)(\mathbb{S}^1)$, hence $\pi(F) = (C, \Sigma)$ (see \eqref{pi}).
The elements in the orbit $[F] = \{F\circ \gamma, \textrm{~for~}\gamma \in G\}$ are all possible parameterizations of $(C,  \Sigma)$ of the form $(F\circ \iota, F)$.
For instance in Fig.~2 one can see some parameterized hands with bracelets that are elements of the same orbit.  
The set of orbits of $\mathcal{P}$ under the group $G$ is called the \textit{quotient space} and will be denoted by $\mathcal{P}/G$. 
\begin{figure*}[ht!]
\includegraphics[width = \textwidth]{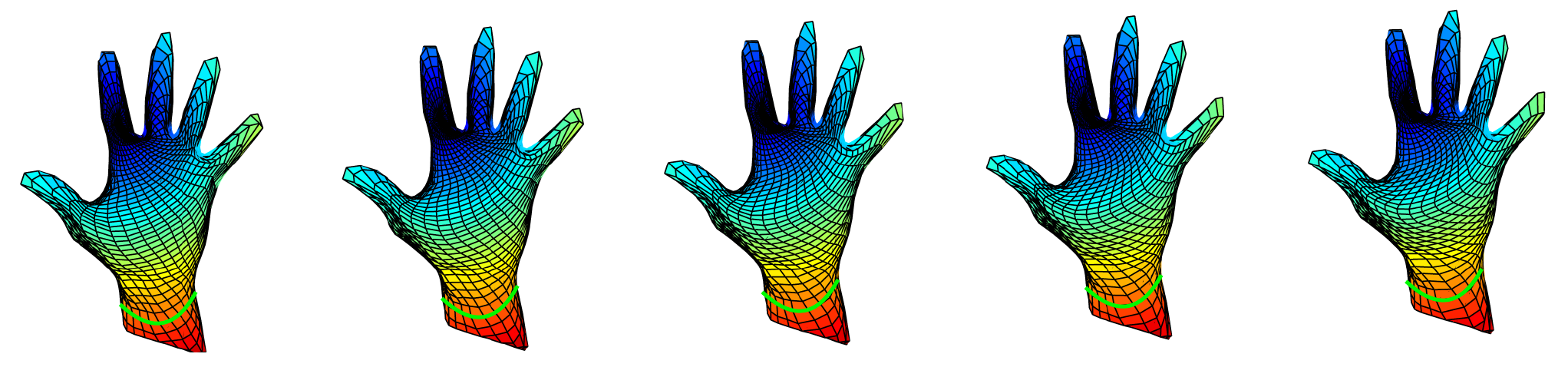}
\caption{Examples of elements in the same orbit under the group of reparameterizations.}
\end{figure*}
\begin{proposition}
The \textit{shape space}  $\mathcal{F}$ is  isomorphic to the quotient space of the pre-shape space  
$\mathcal{P}$ by the shape-preserving group $G = \operatorname{Diff}(\mathbb{S}^2; \iota)$~:
$$\mathcal{F} = \mathcal{P}/G.$$
\end{proposition}
%It is important to note that
The shape space $\mathcal{F} = \mathcal{P}/G$ is a smooth manifold and the 
canonical projection $\pi~:\mathcal{P}\rightarrow \mathcal{F}$, $F\mapsto [F]$ is a submersion (see for instance \cite{VizHal}).
The kernel of the differential of this projection is called the \textit{vertical space}.
It is the tangent space to the orbit of $F\in \mathcal{P}$ under the action of the group $G$.

\begin{proposition}\label{ver}
The vertical space $\operatorname{Ver}_F$ of $\pi$ at some embedding $F\in\mathcal{P}$ is the space of vector fields $X_F\in \mathcal{C}^{\infty}(\mathbb{S}^2, \mathbb{R}^3)$ such that the deformation vector field $X_F\circ F^{-1}$ is tangent to the 
surface $\Sigma:= F(\mathbb{S}^2)$ and such that the restriction of $X_F\circ F^{-1}$ to the curve $C:= F\circ\iota(\mathbb{S}^1)$ % the equator  $\iota: \mathbb{S}^1 \hookrightarrow \mathbb{S}^2$ 
is tangent to $C$.
\end{proposition}

\begin{definition}\label{definition1}
The normal bundle $\operatorname{Nor}$ is the vector bundle over the pre-shape space $\mathcal{P}$, whose fiber over an embedding $F$ is the quotient vector space 
\begin{equation}
\operatorname{Nor}_F := T_F\mathcal{P}/\operatorname{Ver}_F.
\end{equation}
\end{definition}

\begin{proposition}\label{proposition3}
The right action of $G$ on $\mathcal{P}$ induces an action on $T\mathcal{P}$  which preserves the vertical bundle, hence it descends to an action on $\operatorname{Nor}$ by vector bundle homomorphisms.  The quotient bundle $ \operatorname{Nor}/G$ can be identified with the tangent bundle $T\mathcal{F}$. 
\end{proposition}

Consider a nonlinear flag $(C, \Sigma)$. Let us denote by $\nu$ the unit normal vector field on the oriented surface $\Sigma$, and by $t$ the unit vector field tangent to the oriented curve $C$. Set $n := \nu \times t$ the unit normal to the curve $C$ contained in the tangent space to the surface $\Sigma$. The triple $(t, n,\nu)$ is an orthonormal frame along $C$, called the \textit{Darboux frame}. We will denote by $\langle\cdot, \cdot\rangle$ the Euclidian scalar product on $\mathbb{R}^3$.

\begin{figure*}[ht!]
\begin{center}
\includegraphics[width = 0.4 \linewidth]{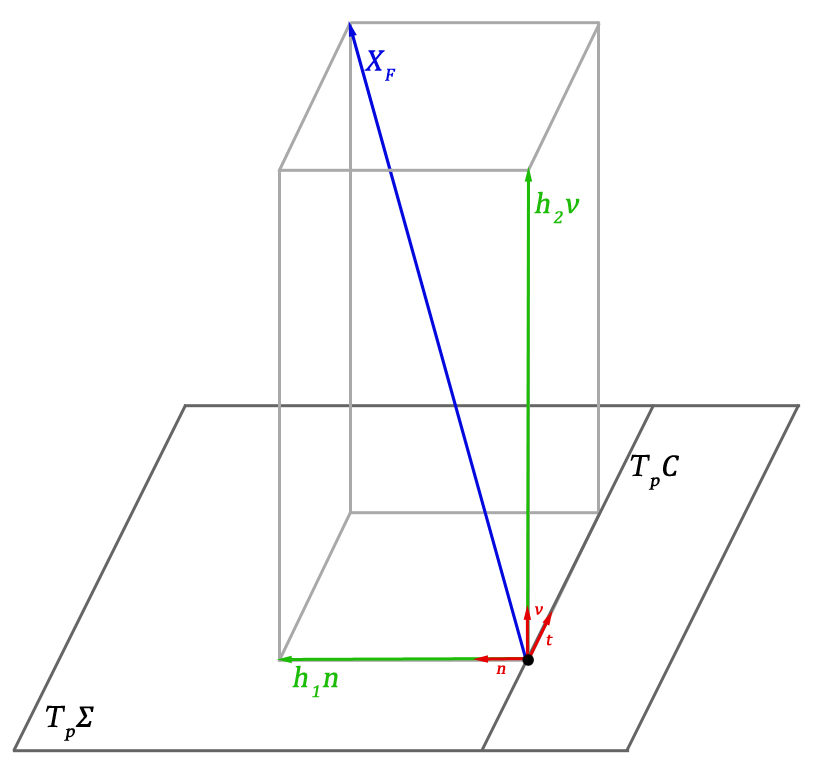}
\caption{Deformation vector field and Darboux frame}
\end{center}
\end{figure*}\label{darb}

\begin{theorem}\label{identification}
Let $F$ be a parameterization of $(C, \Sigma)$. 
Consider the linear surjective map 
\begin{equation}\label{prosec}
\Psi_F~: T_F\mathcal{P} \simeq \mathcal{C}^{\infty}(\mathbb{S}^2, \mathbb{R}^3) \rightarrow \mathcal{C}^{\infty}(C)\times \mathcal{C}^{\infty}(\Sigma),
\end{equation}
which maps $X_F\in T_F\mathcal{P}$ to $(h_1, h_2)$ defined by 
\begin{equation}
\begin{array}{l}
h_1 := \langle (X_F\circ \iota)\circ(F\circ\iota)^{-1}, n\rangle \in \mathcal{C}^{\infty}(C),\\
\\
h_2 := \langle X_F \circ F^{-1}, \nu\rangle \in \mathcal{C}^{\infty}(\Sigma).
\end{array}
\end{equation}
Then the kernel of $\Psi_F$ is 
  the vertical subspace $\operatorname{Ver}_F$, hence $\Psi_F$ defines a map from the quotient space $\operatorname{Nor}_F = T_F\mathcal{P} / \operatorname{Ver}_F$ into  $\mathcal{C}^{\infty}(C)\times \mathcal{C}^{\infty}(\Sigma)$.  The resulting bundle map $\Psi$ is $G$-invariant providing an isomorphism between the tangent space $T_{(C, \Sigma)}\mathcal{F}$ and $\mathcal{C}^{\infty}(C)\times \mathcal{C}^{\infty}(\Sigma)$.
\end{theorem}

\begin{proof}
Consider $X_F$ such that $\Psi_F(X_F) = 0$. Since $h_2 = 0$, $X_F\circ F^{-1}$ is a vector field tangent to $\Sigma$. Since $h_1 = 0$, the restriction of $X_F\circ F^{-1}$ to the curve $C$ given by 
$$
(X_F\circ \iota)\circ(F\circ\iota)^{-1}~:C\rightarrow \mathbb{R}^3
$$
is tangent to $\Sigma $ and orthogonal to $n$, hence it is tangent to $C$.
Thus, by Proposition~\ref{ver},  the kernel of $\Psi$ is exactly $ \operatorname{Ver}_F$.

 Let us show that $\Psi$ is $G$-invariant, i.e. that for $\gamma\in G$,
 \begin{equation}\label{psiinv}
 \Psi_F(X_F) = \Psi_{F\circ\gamma}(X_F\circ\gamma).
 \end{equation}
 One has $\pi(F\circ\gamma) = \pi(F) = (C, \Sigma)$. Moreover  the normal vector fields $\nu~:\Sigma\rightarrow \mathbb{R}^3$ and $n~: C\rightarrow \mathbb{R}^3$ do not depend on the parameterizations of $\Sigma$ and $C$.
 For $\gamma\in G$, we have 
 $$
 (X_F\circ \gamma) \circ (F\circ\gamma)^{-1} = X_F\circ F^{-1}
 $$
 as deformation vector fields on $\Sigma$. 
 On the other hand, using the fact that $\gamma\circ\iota = \iota\circ \bar{\gamma}$ for some $\bar{\gamma}\in  \operatorname{Diff}(\mathbb{S}^1)$, we get
 $$
 \begin{array}{ll}
 (X_F\circ \gamma\circ\iota) \circ (F\circ\gamma\circ\iota)^{-1} & =  (X_F\circ \iota\circ\bar{\gamma}) \circ (F\circ\iota\circ\bar{\gamma})^{-1} \\
& =  (X_F\circ \iota) \circ (F\circ\iota)^{-1}
\end{array}
 $$
 as deformation vector fields on $C$.
The invariance property~\eqref{psiinv} follows.

The projection $\pi~:\mathcal{P}\rightarrow \mathcal{F}= \mathcal{P}/G$ is a principal $G$-bundle, hence the $G$-action preserves the vertical bundle $\operatorname{Ver}$. It  induces a well-defined $G$-action on $\operatorname{Nor}$~: for  $\gamma \in G$,  the class $[X_F] \in \operatorname{Nor}_F$ is mapped to the class $[X_F\circ \gamma] \in \operatorname{Nor}_{F\circ\gamma}$. By $G$-invariance of $\Psi$, we get a well-defined map on $\operatorname{Nor}/G$ which maps isomorphically the tangent space $T_{(C, \Sigma)}\mathcal{F}$ into $ \mathcal{C}^{\infty}(C)\times \mathcal{C}^{\infty}(\Sigma)$.
\end{proof}

\begin{remark}{\rm For the pre-shape space of embedded surfaces, there is a natural section of the projection $T_F\mathcal{P}\to \mathcal{C}^\infty(\Sigma)$,
given by variations that are in the direction of the normal vector field $\nu$ to the surface $\Sigma$. 
In this case, the Euclidean metric on $\mathbb{R}^3$ induces a connection on the principal $\Diff(S^2)$-bundle 
$\P\to\S$, where $\S$ denotes the shape space of surfaces.
Similar sections for the projection \eqref{prosec} are not available for our shape spaces.
In other words, there is no natural principal connection on our principal $G$-bundle $\P\to\F$.}
\end{remark}

 \section{Riemannian metrics on shape spaces of nonlinear flags}
% {Riemannian Metric on nonlinear flags}
 
 As in \cite{Tum1}, we endow the pre-shape space of parameterized curves and surfaces with a family of gauge invariant metrics which descend to a family of Riemannian metric on shape spaces of curves and surfaces. The construction is explained in subsection~\ref{gauge_invariant_section}. The Riemannian metrics on parameterized curves and surfaces used in this construction are the elastic metrics given in subsection~\ref{riemannian_metrics}. The expression of the Riemannian metrics obtained on nonlinear flags in terms of the geometric invariants of curves and surfaces is given in subsection~\ref{sec:geometric-interp}.
 
\subsection{Procedure to construct the Riemannian metrics}\label{gauge_invariant_section}

In order to construct a Riemannian metric on the space $\mathcal{F}$ of nonlinear flags,  we proceed as follows:
\begin{enumerate}
\item we embed our shape space $\F$ of surfaces decorated with curves in the Cartesian product $\S_1\x\S_2$,
 where $\S_1$ denotes the shape space of curves and $\S_2$ the shape space of surfaces.
 \item we choose a family $g^{a, b}$ of $\operatorname{Diff}^+(\mathbb{S}^1)$-invariant metrics on the space of parameterized curves $\P_1$ (equation~\eqref{eqn:elastic-metric1}) 
 \item the family $g^{a, b}$ defines a family of Riemannian metrics on the shape space of curves $\S_1$ by 
 restricting to the normal variations of curves 
\item we choose a family $g^{a', b', c'}$ of $\operatorname{Diff}^+(\mathbb{S}^2)$-invariant metrics on the space of parameterized surfaces $\P_2$(equation~\eqref{eqn:elastic-metric2}).
\item the family $g^{a', b', c'}$ defines a family of Riemannian metrics on the shape space of sufaces $\S_2$ by 
 restricting to the normal variations of surfaces  
\item the product of these metrics is then restricted to $\F$ using the characterization of the tangent space to $\F$ given in Theorem~\ref{identification}.
\end{enumerate}
%The advantage of this procedure is that we do not need to compute the horizontal space of the Riemannian metrics involved. Instead we use the vector bundle $\operatorname{Nor}$ (Definition~\ref{definition1}) and the identification of its quotient $\operatorname{Nor}/G$ by the group of reparameterizations with the tangent space to $\mathcal{F}$ (Proposition~\ref{proposition3}).
\begin{remark}{\rm 
  An equivalent procedure is to  pull back to $\P$, via $F\mapsto ( F\o\io,F)\in\P_1\x\P_2$, 
 the sum of the gauge invariant elastic metrics on the preshape space $\P_1$  for curves and $\P_2$ for surfaces. 
 The result is gauge invariant under $G$, so it descends to a Riemannian metric on the shape space $\F$.
 \begin{equation}\label{E:red.diag}
	\vcenter{
	\xymatrix{
		F\in\P\ \ar[d]_{G}\ar@{^(->}[r]&\P_1\times\P_2\ni(F\o\io,F)\ar[d]^{\Diff(\mathbb{S}^1)\x\Diff(\mathbb{S}^2)}
		\\
		(C,\Si)\in\F\ \ar@{^(->}[r]&\S_1\times \S_2\ni(C,\Si).
	}}
\end{equation}
%- It seems that it leads to the same metric on $\F$. In energy language we have $\E'+\E''$ on $\P_1\x\P$,
%which descends (is truncated) to $E'+E''$ on $\S_1\x\S$, which we then restrict to the energy $\G$ on $\F$. 
%But we could also first restrict  the elastic energy $\E'+\E''$ on $\P_1\x\P$ to $\P$ via the inclusion in the first row,
%which descends (is truncated) to the energy $\G$ on $\F$.}
}
\end{remark}

\subsection{Elastic metrics on manifolds of parameterized curves and surfaces}\label{riemannian_metrics}

The family of Riemannian metric measuring deformations of curves that we will use is the family of $\operatorname{Diff}^+(\mathbb{S}^1)$-invariant elastic metrics on parameterized curves $\P_1$ introduced in \cite{Mio}: 
\begin{equation}\label{elasticmetricdef}
g^{a,b}_{f}(h_1,h_2) = \int_C \big[ a(D_sh_1^{\parallel})(D_sh_2^{\parallel}) + b(D_sh_1^{\perp})(D_sh_2^{\perp})\big] \, d\ell,\\
\end{equation}
where $f\in \mathcal{P}_1$ is a parameterization of the curve $C$, $h_i\in T_f\mathcal{P}_1$ are tangent vectors to the space of parameterized curves, $d\ell = \lVert \dot{f}(t)\rVert dt$,
$D_sh(t) = \tfrac{\dot{h}(t)}{\lVert \dot{f}(t)\rVert}$  is the arc-length derivative of the variation $h$,  $D_sh^{\parallel} = \langle D_sh, t\rangle$ is the component along the unit tangent vector field $t = \tfrac{\dot{f}}{\lVert \dot{f}\rVert}$ to the curve, $D_sh^{\perp} = D_sh - \langle D_sh, t\rangle t$ is the component orthogonal to the tangent vector $t$. Here the $a$-term  measures streching of the curve, while the $b$-term  measures its bending.

Let $\delta f$ denote a  perturbation of a parametrized curve $f:\mathbb{S}^1\to\mathbb{R}^3$, and 
let $(\delta r, \delta {t})$ denote
the corresponding variation of the speed $r := \|\dot{f}(t)\|$ and of the unit tangent vector field $t$. It is easy to check that the squared norm of $\delta f$ for the metric \eqref{elasticmetricdef} reads:
 \begin{equation}\label{eqn:elastic-metric1}
 \E'_f( \delta f)   := g^{a,b}_{f}(\delta f, \delta f) 
%=  \int_{\mathbb{S}^1}  \left\{ a\left(\tfrac{\de r}{r}\right)^2+b |\delta\!  {t}|^2 \right\}d\ell.
=a\int_{\mathbb{S}^1} 
 \left(\frac{\de r}{r}\right)^2d\ell
+b\int_{\mathbb{S}^1}  |\delta {t}|^2 d\ell.
\end{equation}

The family of Riemannian metrics measuring deformations of surfaces that we will use is the family of $\operatorname{Diff}^+(\mathbb{S}^2)$-invariant metrics introduced in \cite{JermynECCV2012} and called elastic metrics. 
Let $\delta F$ denote a  perturbation of a parametrized surface $F$, and 
let $(\delta g, \delta {\nu})$ denote
the corresponding perturbation of the induced metric $g = F^*\langle\cdot,\cdot\rangle_{\mathbb{R}^3}$ and of the unit normal vector field $\nu$. Then the squared norm of $\de F$, namely $ g^{a', b', c'}_F(\delta F, \delta F)$, is:
 \begin{equation}\label{eqn:elastic-metric2}
 \E''_F( \delta F) 
=  %\int_{\mathbb{S}^2} 
 %ds |g|^{\frac{1}{2}}  \left\{ a\operatorname{Tr} (g^{-1}\delta\! g)_{\textbf{0}} ^2 + b \operatorname{Tr}(g^{-1}\delta\! g)^2  + c |\delta\!  {\nu}|^2 \right\}dA,
 a'\int_{\mathbb{S}^2} \operatorname{Tr} (g^{-1}\delta g)_{\textbf{0}} ^2dA+b'\int_{\mathbb{S}^2}\operatorname{Tr}(g^{-1}\delta g)^2  dA +c'\int_{\mathbb{S}^2}  |\delta  {\nu}|^2 dA
 \end{equation}
 where $B_{\textbf{0}} $ is the traceless part of a $2\times 2$-matrix $B$ defined as $B_{\textbf{0}}  = B -\frac{\operatorname{Tr}(B)}{2}I_{2\times2}$. 
 The $a'$-term  measures area-preserving changes in the induced metric $g$, the $b'$-term  measures changes in the area of patches, and the $c'$-term measures bending.

 \subsection{Geometric expression of the Riemannian metrics on manifolds of decorated surfaces} \label{sec:geometric-interp}
 
 In this subsection we restrict the reparametrization invariant metrics \eqref{eqn:elastic-metric1} and \eqref{eqn:elastic-metric2} to normal variations.
 This allows us to express them with the help of the principal curvatures $\ka_1$ and $\ka_2$ of the surface,
 geodesic and normal curvatures $\ka_g,\ka_n$ of the curve on the surface, as well as its geodesic torsion $\ta_g$.
 We recall the identities involving  $(t, n,\nu)$, the Darboux frame: $\dot t=r(\ka_g n+\ka_n\nu)$,
 $\dot n=r(-\ka_g t+\ta_g\nu)$ and $\dot\nu=r(-\ka_n t-\ta_g n)$.
 For functions $h$ on the curve we will use the arc-length derivative $D_sh=\dot h/r$, because it is invariant under reparametrizations. 
 
Moreover, we split the $b$-term in \eqref{eqn:elastic-metric1} into two terms in order to put different weights on the variations along $\nu$ and $n$.
 % \todo{We will use the arclength derivative of a function $h\in \C^\infty(C)$, given by $D_sh=\frac{\dot h}{r}$ with  $r$ denoting the velocity.}
 This leads to the following result~:
 
 \begin{theorem}\label{theorem_expression_metric}
The gauge invariant elastic metrics for parameterized curves respectively surfaces lead to
 a 6-parameter family of Riemannian metrics on the shape space of embedded surfaces decorated with curves:
  \begin{align}\label{geometric_expression}
 %\begin{array}{lll}
 \G_{(C,\Sigma)}(h_1,h_2)&=
 a_1\int_C(h_1\ka_g+{h_2}|_C\ka_n)^2d\ell
 &+ 
a_2\int_{\Sigma}(h_2)^2(\ka_1-\ka_2)^2dA\nonumber\\
& +b_1\int_C(D_sh_1-{h_2}|_C\tau_g)^2d\ell
 &+
 b_2\int_{\Sigma} (h_2)^2(\ka_1+\ka_2)^2dA\nonumber\\
&+c_1\int_C(D_s({h_2}|_C)+h_1\tau_g)^2d\ell
&+
c_2\int_{\Sigma} |\nabla h_2|^2 dA,
%\end{array}
 \end{align}
 for $h_1\in \C^\infty(C)$ and $h_2\in \C^\infty(\Sigma)$.
 \end{theorem}
 
 \begin{proof}
Let $(t,n,\nu)$ be the Darboux frame along the curve $C\subset\Si$.
The normal vector field $h_1n+(h_2|_C)\nu$  to the curve $C$ encodes the variation of the curve which doesn't leave the surface $\Si$.
Using the Lemma \ref{delt}, we obtain the following expression for the elastic metric \eqref{eqn:elastic-metric1} restricted to this normal variation:
\begin{align*}
\E'_F(h_1n+(h_2|_C)\nu)&= a\int_{C} (h_1\ka_g+{h_2}|_C\ka_n)^2d\ell\\
&+b\int_C\left((D_sh_1-{h_2}|_C\tau_g)^2+(D_s({h_2}|_C)+h_1\tau_g)^2\right)d\ell.
\end{align*}
We will split the $b$-term in two parts, thus obtaining a 3-parameter family of metrics, namely 
$$a_1\int_{C} (h_1\ka_g+{h_2}|_C\ka_n)^2d\ell+b_1\int_C(D_sh_1-{h_2}|_C\tau_g)^2d\ell+c_1\int_C(D_s({h_2}|_C)+h_1\tau_g)^2d\ell.$$

The normal vector field $h_2\nu$ to the surface $\Sigma$ encodes the variation of the surface.
%, with $h_2\in\mathcal{C}^{\infty}(\Si)$ and $\nu$ the unit normal vector field. 
Using Eqn.~(12) in \cite{Tum1} the 
elastic metric \eqref{eqn:elastic-metric2} restricted to this normal variation is given by the following geometric expression~:
\begin{equation}\label{expression_metric}
 \E''_F( h_2\nu)=2a\int_{\Si}  (h_2)^2(\kappa_1 - \kappa_2)^2 dA  + 4b \int_{\Si}  (h_2)^2(\kappa_1+\kappa_2)^2dA
 + c\int_{\Si}\,\,  |\nabla h_2|^2 dA.
\end{equation}
Here we use the fact that $g^{-1}\de g=-2h_2L$, where $L$ is the shape operator of the surface, as well as the identity $|\de\nu|=|\nabla h_2|$, where $\nabla$ denotes the gradient with respect to the induced metric on the surface, by Lemma \ref{deltsurf}.

Renaming the parameters and adding  to this elastic metric for the surface  the elastic metric for the curve on the surface obtained above
leads to the 6-parameter family of elastic metrics for the shape space $\F$.
\end{proof}

\begin{remark}{\rm Assuming that the functions $h_1,h_2$ are constant along the curve $C$,
the $b_1$-term becomes $\int_C({h_2}|_C\tau_g)^2d\ell$ and encodes the variation of the curve normal to the surface (variation together with the surface)
while the $c_1$ term becomes $\int_C(h_1\tau_g)^2d\ell$ and encodes the normal variation of the curve inside the surface.}
\end{remark}

\begin{lemma}\label{delt}
Given the normal variation $\de f=h_1 n+h_2|_C \nu$ of the parametrized curve $f=F\circ\io$ on the parametrized surface $F$,
the variation of the speed $r$ and of the unit tangent vector field $t$ are
\begin{gather*}
\de r=-r(h_1\ka_g+h_2|_C\ka_n)\\
\de t=(\tfrac{1}{r}\dot h_1-h_2|_C\ta_g)n+(\tfrac{1}{r}\dot h_2|_C+h_1\ta_g)\nu
\end{gather*}
\end{lemma}

\begin{proof}
For $f_\ep=f+\ep(h_1n+h_2\nu)+O(\ep^2)$, we get
\[
r_\ep^2=r^2-2\ep r^2(h_1\ka_g+h_2|_C\ka_n)+O(\ep^2)
\]
using the well known identities $\dot n=r(-\ka_g t+\ta_g\nu)$ and $\dot\nu=r(-\ka_n t-\ta_g n)$.
Thus $2r\de r=-2r^2(h_1\ka_g+h_2|_C\ka_n)$, hence the first identity.
We use it in the computation of the variation of the unit tangent goes as follows:
\begin{align*}
\de t&=\tfrac{1}{r}\de\dot f-\tfrac{\de r}{r}t=\tfrac{1}{r}(-r(h_1\ka_g+h_2|_C\ka_n)t\\
&+(\dot h_1-rh_2|_C\ta_g)n+(\dot h_2|_C+rh_1\ta_g)\nu)+(h_1\ka_g+h_2|_C\ka_n)t\\
&=(\tfrac{\dot h_1}{r}-h_2|_C\ta_g)n+(\tfrac{\dot h_2|_C}{r}+h_1\ta_g)\nu,
\end{align*}
hence the second identity.
\end{proof}

\begin{lemma}\label{deltsurf}
Given the normal variation $\de F=h\nu$ of the parametrized surface  $F$,
the variation of  the unit normal vector field $\nu$ and the gradient of $h$
with respect to the induced metric on the surface have the same norm.
\end{lemma}

\begin{proof}
Let $(u,v)$ denote coordinates on $\mathbb{S}^2$ and let $F_u,F_v$ denote the partial derivatives of $F$ (and similarly for $h$). 
Then, as in \cite{Tum2}, we get the variation
\begin{equation*}
\de \nu=-(h_u,h_v)g^{-1}(F_u,F_v)^\top.
\end{equation*}
On the other hand $\nabla h=g^{-1}(h_u,h_v)^\top$.
Now we compute
\begin{align*}
|\de\nu |^2&=(h_u,h_v)g^{-1}(F_u,F_v)^\top(F_u,F_v)g^{-1}(h_u,h_v)^\top\\
&
=(h_u,h_v)g^{-1}(h_u,h_v)^\top=(\nabla h)^\top g\nabla h=|\nabla h|^2,
\end{align*}
using the fact that $g=(F_u,F_v)^\top(F_u,F_v)$.
\end{proof}

 \section{Conclusion}
 
 In this paper, we identify the tangent spaces to nonlinear flags consisting of surfaces of genus zero decorated with a simple curve.
 We use gauge invariant metrics on parameterized curves and surfaces to endow the space of nonlinear flags with a family of Riemannian metrics, whose expression is given in terms of geometric invariants of curves and surfaces.

\subsubsection{Acknowledgements} The second author is supported by FWF grant I 5015-N.
The first and the third authors are  supported
by a grant of the Romanian Ministry of Education and Research, CNCS-UEFISCDI, project
number PN-III-P4-ID-PCE-2020-2888, within PNCDI III.


\end{document}